\documentclass[11pt]{article}

\usepackage{amsmath,amssymb,amsbsy,amsfonts,amsthm,latexsym,
            amsopn,amstext,amsxtra,euscript,amscd,color,breqn,mathrsfs}

\PassOptionsToPackage{hyphens}{url}\usepackage{hyperref}
    
\usepackage[capbesideposition=outside,capbesidesep=quad]{floatrow}

\restylefloat{table}
            
\usepackage{multirow,caption}

\newfont{\teneufm}{eufm10}
\newfont{\seveneufm}{eufm7}
\newfont{\fiveeufm}{eufm5}

\usepackage{systeme}

\usepackage{xcolor}
\usepackage{tikz}
\usetikzlibrary{shapes}
\usetikzlibrary{decorations.markings}

\usepackage{pstricks}
\usetikzlibrary{shapes,arrows}
 \usepackage{pstricks-add}
\usepackage{epsfig}
 \usepackage[space]{grffile} 
 \usepackage{etoolbox} 
 \makeatletter 
 \patchcmd\Gread@eps{\@inputcheck#1 }{\@inputcheck"#1"\relax}{}{}
 \makeatother

\newtheorem{thm}{Theorem}
\newtheorem{lem}[thm]{Lemma}
\newtheorem{cor}[thm]{Corollary}

\newtheorem{rem}[thm]{Remark}

\newcommand{\Tr}{{\rm Tr}}
\newcommand{\Trn}{{\rm Tr}_n}
\newcommand{\Trd}{{\rm Tr}_d}

\def\+{\oplus}

\def\cN{{\mathcal N}}

\def\cP{{\mathcal P}}

\def\F{{\mathbb F}}

\def\00{{\bf 0}}
\def\11{{\bf 1}}
\def\+{\oplus}

\def\\{\cr}
\def\({\left(}
\def\){\right)}

\newcommand{\cardinality}[1]{\# #1}

\usepackage[colorinlistoftodos,prependcaption,textsize=tiny]{todonotes}

\providecommand{\newoperator}[3]{%
  \newcommand*{#1}{\mathop{#2}#3}}

\newoperator{\FD}{\mathrm{FD}}{\nolimits}

\begin{document}
\title{\bf Characters, Weil sums and  $c$-differential uniformity with an application to  the perturbed Gold function}
\author{\Large Pantelimon~St\u anic\u a
$^1$ \and \Large Constanza Riera$^2$, Anton Tkachenko$^2$
\\ \\
$^1$Applied Mathematics Department, \\
Naval Postgraduate School, Monterey, USA; 
{\tt pstanica@nps.edu}\\
$^2$Department of Computer Science,\\
 Electrical Engineering and Mathematical Sciences,\\
   Western Norway University of Applied Sciences,
  5020 Bergen, Norway;\\ {\tt \{csr, atk\}@hvl.no}\\}

\maketitle

\begin{abstract}
Building upon the observation that the newly defined~\cite{EFRST20} concept of $c$-differential uniformity  is not invariant under EA or CCZ-equivalence~\cite{SPRS20}, we showed in~\cite{SG20} that adding  some appropriate linearized monomials increases the $c$-differential uniformity of  the inverse function, significantly, for some~$c$. We continue that investigation here.
First, by analyzing the involved equations, we find bounds for the uniformity of the Gold function perturbed by a single monomial, exhibiting the discrepancy we previously observed on the inverse function.
Secondly, to treat the general case of perturbations via any linearized polynomial, we use characters in the finite field to express all entries in the $c$-Differential Distribution Table (DDT) of an $(n,n)$-function on the finite field $\F_{p^n}$, and further, we use that  method to find explicit expressions for all entries of the $c$-DDT of the perturbed Gold function (via an arbitrary linearized polynomial).
\end{abstract}


\section{Introduction and basic definitions}

  Motivated by the challenge of~\cite{BCJW02}, who  extended the differential attack on some ciphers  by using a new type of differential,   we  defined in~\cite{EFRST20} a new differential and difference distribution table, in any characteristic, along with the corresponding perfect/almost perfect $c$-nonlinear functions, etc., (unbeknown to us, and developed independently, this is a generalization of the recent~\cite{BT} concept of quasi planarity: a quasi planar function is simply  a perfect $c$-nonlinear function for $c=-1$). We later extended the notion of boomerang connectivity table  in~\cite{S20} and characterized some of the known perfect nonlinear functions and the inverse function through this new concept. In~\cite{EFRST20,SPRS20,RS20,YZ20} various characterizations of the $c$-differential uniformity were found, and some of the known perfect and almost perfect nonlinear functions have been investigated.  An approach on boomerang uniformity based upon Weil sums and characters was developed in~\cite{S20_Weil}.  We will take a similar approach in this paper on $c$-differential uniformity, which has the advantage of providing some character expressions for all entries in the  $c$-Differential Distribution Table (defined below).

 While we only introduce here only some needed notation on Boolean (binary, $p=2$) and $p$-ary functions (where $p$ is an odd prime),  the reader can consult~\cite{Bud14,CH1,CH2,CS17,MesnagerBook,Tok15} for more   on cryptographic Boolean functions and their properties.

Let $p$ be a prime number and $n$ be a positive integer $n$. We let $\F_{p^n}$ be the  finite field with $p^n$ elements, and $\F_{p^n}^*=\F_{p^n}\setminus\{0\}$ be the multiplicative group; for $a\neq 0$, we often write $\frac{1}{a}$ to mean the inverse of $a$ in the multiplicative group. We let $\F_p^n$ be the $n$-dimensional vector space over $\F_p$.  We use $\cardinality{S},\bar S$ to denote the cardinality of a set $S$, respectively, the complement of $S$ in a superset (usually, $\F_{p^n}$), which will be clear from the context. Also, for a complex number  $z$, we denote by $\bar z$   its complex conjugate.

We call a function from $\F_{p^n}$ (or $\F_p^n$) to $\F_p$  a {\em $p$-ary  function} on $n$ variables. For positive integers $n$ and $m$, any map $F:\F_{p^n}\to\F_{p^m}$ (or, $\F_p^n\to\F_p^m$)  is called a {\em vectorial $p$-ary  function}, or {\em $(n,m)$-function}. When $m=n$, $F$ can be uniquely represented as a univariate polynomial over $\F_{p^n}$ 
 of the form
$F(x)=\sum_{i=0}^{p^n-1} a_i x^i,\ a_i\in\F_{p^n}$,
whose {\em algebraic degree}   is then the largest Hamming weight of the exponents $i$ with $a_i\neq 0$. 
We let $\Trn:\F_{p^n}\to \F_p$ be the absolute trace function, given by $\displaystyle \Trn(x)=\sum_{i=0}^{n-1} x^{p^i}$. Also, $\Trd(x)=\sum_{i=0}^{\frac{n}{d}-1} x^{p^{di}}$ is the relative trace from $\F_{p^n}\to\F_{p^d}$, where $d\,|\,n$.
 
%

 
For a $p$-ary $(n,m)$-function   $F:\F_{p^n}\to \F_{p^m}$, and $c\in\F_{p^m}$, the ({\em multiplicative}) {\em $c$-derivative} of $F$ with respect to~$a \in \F_{p^n}$ is the  function
\[
 _cD_{a}F(x) =  F(x + a)- cF(x), \mbox{ for  all }  x \in \F_{p^n}.
\]

For an $(n,n)$-function $F$, and $a,b\in\F_{p^n}$, we let the entries of the $c$-Difference Distribution Table ($c$-DDT) be defined by ${_c\Delta}_F(a,b)=\cardinality{\{x\in\F_{p^n} : F(x+a)-cF(x)=b\}}$. We call the quantity
\[
\delta_{F,c}=\max\left\{{_c\Delta}_F(a,b)\,:\, a,b\in \F_{p^n}, \text{ and } a\neq 0 \text{ if $c=1$} \right\}\]
the {\em $c$-differential uniformity} of~$F$. If $\delta_{F,c}=\delta$, then we say that $F$ is differentially $(c,\delta)$-uniform (or that $F$ has $c$-uniformity $\delta$). If $\delta=1$, then $F$ is called a {\em perfect $c$-nonlinear} ({\em PcN}) function (certainly, for $c=1$, they only exist for odd characteristic $p$; however, as proven in~\cite{EFRST20}, there exist PcN functions for $p=2$, for all  $c\neq1$). If $\delta=2$, then $F$ is called an {\em almost perfect $c$-nonlinear} ({\em APcN}) function. 
When we need to specify the constant $c$ for which the function is PcN or APcN, then we may use the notation $c$-PN, or $c$-APN.
It is easy to see that if $F$ is an $(n,n)$-function, that is, $F:\F_{p^n}\to\F_{p^n}$, then $F$ is PcN if and only if $_cD_a F$ is a permutation polynomial. For $c=1$, we recover the classical derivative, PN, APN, etc., differential uniformity and DDT.

The rest of the paper is organized as follows.  Section~\ref{sec2_0} gives bounds for the $c$-differential uniformity for the Gold function perturbed by a single monomial. Section~\ref{sec2} gives  a general theorem describing the entries of the $c$-DDT via characters in the finite field. Section~\ref{sec3}  investigates  $c$-DDT entries for a perturbation via an arbitrary linearized monomial of the Gold function, for $p$ odd. 
Section~\ref{sec4}  completes the investigation and does the same for $p$ even.
Section~\ref{sec5}   concludes the paper.

 \section{Perturbations of the Gold function via a linearized monomial}
 \label{sec2_0}
 We shall be using, throughout the paper, the following lemma. 
 \begin{lem}[\textup{\cite{Co98_1,EFRST20}}]
\label{lem:gcd}
Let $p,t,n$ be integers greater than or equal to $1$ \textup{(}we take $t\leq n$, though the result can be shown in general\textup{)}. Let $d=\gcd(n,t), e=\gcd(n,2t)$. Then,
\begin{align*}
&  \gcd(2^{t}+1,2^n-1)=\frac{2^e-1}{2^d-1},  \text{ and if  $p>2$, then}, \\
& \gcd(p^{t}+1,p^n-1)=2,   \text{ if $\frac{n}{d}$  is odd},\\
& \gcd(p^{t}+1,p^n-1)=p^d+1,\text{ if $\frac{n}{d}$ is even}.
\end{align*} 
\end{lem}
We showed in~\cite{EFRST20} that the inverse function is PcN for $c=0$,  and it is 2 or~3 depending upon the parameter $c$ (we found precisely those conditions).  In~\cite{SG20} we showed that adding   $x^{2^d}$ to $x^{2^n-2}$, where $d$ is the largest nontrivial divisor of $n$, increases the mentioned $c$-differential uniformity from~$2$ or $3$ (for $c\neq 0,1$) to $\geq 2^{d}+2$ (in the case of the inverse function as used in the Advanced Encryption Function (AES) it is 18). This discrepancy is rather surprising and prompts an   investigation into other well-behaved, under classical differential uniformity, vectorial functions.
 
 In the  result of this section we see that simply adding a linearized monomial to   the Gold function increases significantly the maximum value in its $c$-differential spectrum size. In the following, we take $p$ prime, $n\geq 4$ an integer, and $0\leq t<n$ an integer such that   $a^{p^k-p^t+1}+1$ has a root (and consequently, $\gcd(p^k-p^t+1,p^n-1)$ roots) in the  field $\F_{p^n}$.

 \begin{thm}
 \label{thm:main1}
 Let $p$ be a prime number, $n\geq 4$, $F(x)=x^{p^k+1}$ be the Gold function on $\F_{p^n}$, and $1\neq c\in\F_{p^n}$, $1\leq k<n$ with $\gcd(k,n)=d\geq 1$ and $\frac{n}{\gcd(n,k)}\geq 3$. Then, the $c$-differential uniformity, $\delta_{G,c}$, of $G(x)=F(x)+x^{p^t}$ satisfies
 $\gcd(p^k-p^t+1,p^n-1)+1\leq \delta_{G,c}\leq \max\{p^k+1,p^t\}$;
if  $G(x)=F(x)+x$, or $G(x)=F(x)+x^{p^k}$, then $p^{\gcd(n,k)}+1\leq \delta_{G,c}\leq p^k+1$.
   \end{thm}
 \begin{proof}
 Let $G(x)=x^{p^k+1}+x^{p^t}$. The $c$-differential uniformity equation for $G$ for $c\in\F_{p^n}$ at $(a,b)\in\F_{p^n}\times\F_{p^n}$ is
 \begin{equation}
 \label{eq:cdiff}
 (x+a)^{p^k+1}+(x+a)^{p^t}-c x^{p^k+1}-cx^{p^t}=b.
 \end{equation}
 If $a=0$, the equation becomes
\[
 x^{p^k+1}+  x^{p^t}-\frac{b}{1-c}=0.
\]
Surely, if $b=0$, then $x=0$ and $x^{p^k-p^t+1}+1=0$. The latter equation (under our assumption) has $\gcd(p^k-p^t+1,p^n-1)$ solutions. 
However, if $b\neq 0$, the equation is not easy to handle, unless $t$ has some special forms, which are dealt with below.

 We continue with $G(x)=x^{p^k+1}+x$ and
   use the $c$-differential uniformity equation of $G$ for $c\in\F_{p^n}$ at $(a,b)\in\F_{p^n}\times\F_{p^n}$ from~\eqref{eq:cdiff}, which, when $t=0$ becomes 
 \begin{equation}
 \label{eq:cdiff0}
 (1-c) x^{p^k+1} +a x^{p^k}+(1+a^{p^k} -c) x+a^{p^k+1}+a-b=0.
 \end{equation}
 Clearly, $\delta_{G,c}\leq p^k+1$. We now find a lower bound. 
If $a=0$, the equation becomes
\[
 x^{p^k+1}+  x-\frac{b}{1-c}=0.
\]
If $b=0$, then $x=0$ and $x^{p^k}+1=0$. The latter equation has a unique solution, since $\gcd(p^k,p^n-1)=1$. Thus, if $a=0$, we have two solutions for~\eqref{eq:cdiff0}. If $b\neq 0$, we use the transformation $x=\frac{b}{1-c} y$ and obtain
\begin{equation}
\label{blu:eq}
y^{p^k+1}-B y+B=0,
\end{equation}
where $B=\left(\frac{1-c}{b}\right)^{p^k}$. We now use~\cite[Theorem 5.6]{Bluher04}. We let $Q=p^{\gcd(n,k)}$, so $\F_Q=\F_{p^n}\cap \F_{p^k}$, $m=[\F_{p^n}:\F_{Q}]=\frac{n}{\gcd(n,k)}$. By~\cite[Theorem 5.6]{Bluher04}, we know that there are $\displaystyle \frac{Q^{m-1}-Q}{Q^2-1}$, $\displaystyle \frac{Q^{m-1}-1}{Q^2-1}$, for $m$ even, respectively odd, values of $B$ such that Equation~\eqref{blu:eq} has $Q+1$ solutions.
Let $T$ be the set of all such $B$.  
For any $B\in T$, we let $b\neq 0$ be random and $c= 1-b(B)^{p^{-k}}$. For such choices of parameters, we do get $p^{\gcd(n,k)}+1$ solutions, and so, $\delta_{G,c}\geq p^{\gcd(n,k)}+1$.
  
 We next assume that $a\neq 0$ (while we do not need to consider this case to show our claim, we do treat it here, just to point out that the $c$-DDT may have other entries, not only on the first row, with large values).  
 We now remove the coefficient of $x^{p^k}$ with the transformation 
 \[
  x\mapsto \frac{c a^{p^k+1}-b c+b}{(1-c)
   \left(a^{p^k}-\left(\frac{a}{1-c}\right)^{p^k}+c
   \left(\frac{a}{1-c}\right)^{p^k}-c+1\right)} x-\frac{a}{1-c}
   \]
    and     Equation~\eqref{eq:cdiff} becomes~\eqref{blu:eq}, where now, 
    \begin{equation}
    \label{eq:B}
    B=\frac{(a_1-e^{p^k})^{p^k+1}}{(b_1-ea_1)^{p^k}},\ e=\frac{a}{1-c},a_1=1+\frac{a^{p^k}}{1-c}, b_1=\frac{a+a^{p^k+1}-b}{1-c}.
    \end{equation}
 By~\cite[Theorem 5.6]{Bluher04}, again we have $p^{\gcd(n,k)}+1$ solutions for~\eqref{blu:eq}, for $B$ belonging to a  set $T$ of cardinality $\displaystyle |T|=\frac{Q^{m-1}-Q}{Q^2-1}$, $\displaystyle |T|= \frac{Q^{m-1}-1}{Q^2-1}$, for $m$ even, respectively odd.  Clearly, for $a,b,c$ such that~\eqref{eq:B} holds (there is no need to check if their  existence, since we know they do, from the first part of the proof), then again we have ${_c}\Delta_G(a,b)\geq p^{\gcd(n,k)}+1$, and thus  $\delta_{G,c}\geq p^{\gcd(n,k)}+1$ for those values of $c$.
    
    We next continue with $G(x)=x^{p^k+1}+x^{p^k}$. Equation~\eqref{eq:cdiff} is now
    \begin{equation}
    \label{eq:2ndclaim}
   x^{p^k+1}+\left(1+\frac{a}{1-c}\right)x^{p^k}+ \frac{a^{p^k}}{1-c} x+\frac{a^{p^k+1}+a^{p^k}-b}{1-c}=0.
    \end{equation}
    If $c=a+1$, this equation is now 
    \[
    x^{p^k+1}-a^{p^k-1} x+\left(\frac{b}{a}-a^{p^k}-a^{p^k-1}\right)=0,
    \]
    which is equivalent to $x^{p^k+1}-B x+B=0$, where $\displaystyle B=\frac{a^{p^{2k}-1}}{\left(\frac{b}{a}-a^{p^k}-a^{p^k-1}\right)^{p^k}}$, and this equation can be treated via~\cite[Theorem 5.6]{Bluher04}, as well (observe that, regardless of what $B\neq 0$ is, we can always find $a,b$ such that the previous identity holds: for example, we can take $b=a^{p^k}$, and so, $B=\frac{1}{(-1)^{p^k} a}$).
    If $c\neq a+1$, as we did for the first claim, we remove the coefficient of $x^{p^k}$ by using a transformation 
    \[
  x\mapsto \frac{c a^{p^k+1}-b c+b}{(1-c)
   \left(a^{p^k}-\left(\frac{a}{1-c}+1\right)^{
   p^k}+c
   \left(\frac{a}{1-c}+1\right)^{p^k}\right)} x-\frac{a}{1-c}-1
   \]
    and
    Equation~\eqref{eq:2ndclaim} becomes~$x^{p^k+1}-Bx+B=0$,
    where
     \begin{equation}
    \label{eq:B1}
    B=\frac{(a_1-e^{p^k})^{p^k+1}}{(b_1-ea_1)^{p^k}},\ e=1+\frac{a}{1-c},a_1=\frac{a^{p^k}}{1-c}, b_1=\frac{a^{p^k+1}+a^{p^k}-b}{1-c},
    \end{equation}
    enabling us to use, yet again,~\cite[Theorem 5.6]{Bluher04}, to infer ${_c}\Delta_{G}(a,b)\geq p^{\gcd(n,k)}+1$, and consequently, $\delta_{G,c}\geq p^{\gcd(n,k)}+1$.
The theorem is shown.
\end{proof}
               
\section{Characters and $c$-differential uniformity}
\label{sec2}

We showed in~\cite{S20_Weil} a general theorem expressing the entries in the $c$-Boomerang Connectivity Table  (for all $c\neq 0$) in terms of double Weil sums.  
There is no reason why that is not developed for the $c$-DDT, and we shall do that below. We first show a general theorem that gives {\em all} entries of the $c$-DDT for any function in terms of characters of the corresponding finite field (we will also include $c=1$ in our analysis, since the use of characters does not seem to be a method of choice for classical computation of the DDT). For  the convenience of the reader, we will go through the proof, although it follows in general lines the characters computation for the entries of the boomerang connectivity table method of~\cite{S20_Weil}.

Let $G$ be the Gauss' sum $\displaystyle G(\psi,\chi)=\sum_{z\in\F_q^*} \psi(z)\chi(z)$, where $\chi,\psi$, are additive, respectively, multiplicative characters of $\F_q$, $q=p^n$. Below, we let $\chi_1(a)=\exp\left(\frac{2\pi i\Trn(a)}{p}\right)$ be the principal additive character, and $\psi_k\left(g^\ell\right)=\exp\left(\frac{2\pi ik\ell}{q-1}\right)$ be the $k$-th multiplicative character of $\F_q$, $0\leq k\leq q-2$. 

\begin{thm}
\label{thm:cDU_Char}
Let $F(x)$ be an arbitrary function on $\F_q$, $q=p^n$, $p$ a prime number, and $c\in\F_q^*$. Then, the $c$-Differential Distribution Table entry   at $(a,b)$  is given by 
\[
{_c}\Delta_F(a,b)=1+\frac1{q} \sum_{\alpha\in\F_q^*} \chi_1(-b \alpha)\sum_{x\in \F_q}
\chi_1\left(\alpha \left( F(x+a) -c F(x) \right)\right).
\]
\end{thm}
\begin{proof}
Recall that ${_c}\Delta_F(a,b)$ is the number of solutions in $\mathbb{F}_{q}$, $q=p^n$, for the equation
\begin{equation}
\label{eq:eq3.1}
F(x+a)-c F(x)=b.
\end{equation}
As done in~\cite{S20_Weil},  
we know  that the number  $\cN(b)$  of of solutions $(x_1,\ldots,x_n)\in\F_q^n$, for $b\in\F_{p^m}$ fixed, of an equation $f(x_1,\ldots,x_n)=b$ is
\begin{align*}
\cN(b)
&= \frac{1}{q}\sum_{x_1,\ldots,x_n\in \F_q}\sum_{\alpha\in\F_q} \chi_1\left(\alpha \left( f(x_1,\ldots,x_n)- b\right)\right)\\
&=\frac{1}{q}\sum_{x_1,\ldots,x_n\in \F_q}\sum_{\chi\in \widehat{\F_q}}\chi(f(x_1,\ldots,x_n))\overline{\chi(b)},
\end{align*}
where $\widehat{\F_q}$ is the set of all additive characters of $\F_q$, and $\chi_1$ is the principal additive character of  $\F_q$. For our Equation~\eqref{eq:eq3.1}, we see that the number of solutions for some $a,b$ fixed is therefore
\allowdisplaybreaks
\begin{align*}
&\cN_{a,b;c}
=\frac{1}{q}\sum_{x\in \F_q}\sum_{\alpha\in\F_q} \chi_1\left(\alpha\left(F(x+a)-cF(x)-b \right) \right)  \\
&= \frac{1}{q}\sum_{\alpha\in\F_q} \chi_1(-b\alpha)
\sum_{x\in \F_q} \chi_1\left(\alpha  F(x+a)\right) \chi_1\left(-\alpha c F(x)\right).
\end{align*}
Splitting, based on $\alpha=0$ and $\alpha\neq 0$, we write
\allowdisplaybreaks
\begin{align*}
{_c}\Delta_F(a,b)&=1+\frac1{q} \sum_{\alpha\in\F_q^*} \chi_1(-b \alpha)\sum_{x\in \F_q}
\chi_1\left(\alpha  F(x+a) -\alpha c F(x) \right)\\
&= 1+\frac1{q} \sum_{\alpha\in\F_q^*} \chi_1(-b \alpha) U_\alpha,
\end{align*}
where $U_{\alpha} :=\sum_{x\in \F_q} \chi_1\left(\alpha  F(x+a)-\alpha c F(x)\right)$.
 \end{proof}
 \begin{cor}
 \label{cor:cDU_Char}
  For all $c\in\F_q$, if $a=0$, then  
\[
 U_\alpha=\sum_{x\in \F_q} \chi_1\left(\alpha  (1-c) F(x)\right),
\]
and so,
 \[
{_c}\Delta_F(0,b)=1+\frac1{q}\sum_{\alpha\in\F_q^*} \chi_1(-b \alpha)  \sum_{x\in \F_q} \chi_1\left(\alpha  (1-c) F(x)\right).
 \]
 \end{cor}

 \section{Entries of the $c$-DDT for the perturbed Gold function via a linearized polynomial, $p$ odd}
 \label{sec3}
 
 We now take the particular case of the Gold function $F(x)=x^{p^k+1}$ on $\F_q$, $1\leq k<n$, $q=p^n$, $p$ prime, $n\geq 2$,  perturbed by any linearized polynomial $P(x)=\sum_{i=0}^{n-1} a_i x^{p^i}$, that is, $G(x)=x^{p^k+1}+\sum_{i=0}^{n-1} a_i x^{p^i}$. We fix $c\in\F_q$ (for $c=1$ many of the expressions will simplify significantly, since the term $(1-c)P(x)$ below will disappear) but we kept that case for completeness, since we do not believe there was ever a complete description for the DDT of the Gold function (surely, in this case, in terms of characters). For every $\alpha\in\F_q^*$, we let $A_\alpha=\alpha(1-c)$, 
 $P^*(x)=\sum_{i=0}^{n-1} ((1-c)a_i)^{p^{n-i}} x^{p^{n-i}}$ be the linearized $c$-companion   polynomial for~$P$, 
 and $B_\alpha=\sum_{i=0}^{n-1} (a_i')^{p^{n-i}}$, 
 where $a_i'=\alpha (1-c)a_i=A_\alpha a_i$, for all $0\neq i\neq k$, $a_0'=\alpha\left(a^{p^k}+(1-c)a_0\right)=A_\alpha a_0+\alpha a^{p^k}$ and $a_k'=\alpha (a+(1-c)a_k)=A_\alpha a_k+\alpha a$.

 We next expand 
 \begin{align*}
 G(x+a)-c G(x)&=(1-c) x^{p^k+1} +a^{p^k}x+a x^{p^k}+(1-c)P(x)+a^{p^k+1} +P(a).
 \end{align*}
  Thus, using Theorem~\ref{thm:cDU_Char} and the fact that $\chi_1(y^p)=\chi_1(y)$, for all $y\in\F_q$, implying $\chi_1(a_i' x^{p^i})=\chi_1((a_i')^{p^{n-i}} x)$, we get
 \allowdisplaybreaks
\begin{align*}
{_c}\Delta_G(a,b)&= 1+\frac1{q} \sum_{\alpha\in\F_q^*} \chi_1(-b \alpha) U_\alpha,
\end{align*}
where (we use the notation $P'(a)=P(a)+a^{p^k+1}$)
 \allowdisplaybreaks
\begin{align*}
U_{\alpha} &=\sum_{x\in \F_q} \chi_1\left(\alpha  G(x+a)-\alpha c G(x)\right)\\
&=\sum_{x\in \F_q} \chi_1\left(\alpha\left((1-c) x^{p^k+1} +a^{p^k}x+a x^{p^k}+(1-c)P(x)+ P'(a)\right) \right)\\
&=\chi_1(\alpha  P'(a))  \sum_{x\in \F_q} \chi_1\left(\alpha(1-c) x^{p^k+1}\right)\chi_1\left(\alpha \left(ax^{p^k}+a^{p^k}x+ (1-c)P(x)\right) \right)\\
&=\chi_1(\alpha  P'(a))    \sum_{x\in \F_q} \chi_1\left(A_\alpha x^{p^k+1}+ B_\alpha  x\right).
\end{align*}

Therefore, 
\begin{align*}
{_c}\Delta_G(a,b)&= 1+\frac1{q}   \sum_{\alpha\in\F_q^*} \chi_1\left(\alpha(P'(a)-b)\right)\sum_{x\in \F_q} \chi_1\left(A_\alpha x^{p^k+1}+ B_\alpha  x \right).
\end{align*}

 For general $c$ and $a$ fixed, we now let $X_a\subseteq\F_q^*$ be defined by 
 (the second formulation is obtained by  raising the first one to the $p^k$ power)
 \begin{equation}
 \begin{split}
 \label{eq:X_asol}
 X_a&=\left\{\alpha\in\F_q^*\,:\,  \alpha a^{p^k}+\alpha^{p^{-k}} a^{p^{-k}}  +P^*(\alpha)=0 \right\}\\
 &=\left\{\alpha\in\F_q^*\,:\,  \alpha^{p^k} a^{p^{2k}} +\alpha a+(P^*(\alpha))^{p^k}=0 \right\}.
 \end{split}
 \end{equation}
 (This is the set of $\alpha$'s such that $B_\alpha=0$.)

Next, we let
 \allowdisplaybreaks
\begin{align*}
S_\alpha&= \sum_{x\in \F_q}\chi_1\left(A_\alpha x^{p^k+1}+ B_\alpha  x \right)\\
T_{a,b}&= \sum_{\alpha\in\F_q^*}  \chi_1(\alpha(P'(a)-b))S_\alpha.
\end{align*}

With these notations, we thus obtain
 \allowdisplaybreaks
 \begin{align*}
 T_{a,b}&= \sum_{\alpha\in X_a} \chi_1(\alpha (P'(a)-b))   \sum_{x\in \F_q} \chi_1\left(A_\alpha  x^{p^k+1}\right)\\
 &+ \sum_{\alpha\in \bar X_a} \chi_1(\alpha (P'(a)-b))   \sum_{x\in \F_q} \chi_1\left(A_\alpha  x^{p^k+1}+B_\alpha x\right)\\
 &=:T_1+T_2.
 \end{align*}
  
  We let $\eta=\psi_{(q-1)/2}$ be the quadratic character of $\F_q$ and for some $A,B \in \F_q$, $1\leq k<n$, $d=\gcd(n,k)$, we let $\mathscr{S}_k(A,B)=\sum_{x\in \F_q} \chi_1\left(A x^{p^k+1}+ B x\right)$.
 We  now use~\cite[Theorem 1 and 2]{Co98_1} (we simplify the original statement).  
  \begin{thm}[\textup{\cite{Co98_1}}]
  \label{thm:Co98_1}
  Let $q=p^n$, $1\leq k<n, d=\gcd(n,k)$. The following statements hold\textup{:}
\begin{enumerate}
\item[$(1)$] When  $\frac{n}{d}$ is even $(n=2m)$, then
 \[
\mathscr{S}_k(A,0)=
\begin{cases}
 (-1)^{\frac{m}{d}}\, p^m\, &\text{ if } A^{\frac{q-1}{p^d+1}}\neq (-1)^{\frac{m}{d}}\\ 
 (-1)^{\frac{m}{d}+1}\, p^{m+d}\, &\text{ if } A^{\frac{q-1}{p^d+1}}= (-1)^{\frac{m}{d}}.
\end{cases}
\]
\item[$(2)$] When $\frac{n}{d}$ is odd, then 
 \[
\mathscr{S}_k(A,0)=
\begin{cases}
(-1)^{n-1}\sqrt{q}\, \eta(A) &\text{ if } p\equiv 1\pmod 4\\ 
 (-1)^{n-1}\imath^n \sqrt{q}\, \eta(A)  &\text{ if } p\equiv 3\pmod 4.
\end{cases}
\]
\end{enumerate}
\end{thm}
 Therefore, with $A=A_\alpha,B=B_\alpha=0$ ($c$ is fixed), and even $\frac{n}{d}$ (so, $n=2m$), we obtain
 \begin{equation}
 \begin{split}
 \label{eq:Tb_even}
 T_1&=  (-1)^{\frac{m}{d}} p^m \sum_{\substack{\alpha\in X_a\\ A_\alpha^{\frac{q-1}{p^d+1}}\neq (-1)^{\frac{m}{d}}} } \chi_1(\alpha(P'(a)-b)) \\
 &\qquad + (-1)^{\frac{m}{d}+1}\, p^{m+d}  \sum_{\substack{\alpha\in X_a\\ A_\alpha^{\frac{q-1}{p^d+1}}= (-1)^{\frac{m}{d}}} } \chi_1(\alpha(P'(a)-b)).
 \end{split}
 \end{equation} 
 
 Observe that the equation $\displaystyle A_\alpha^{\frac{q-1}{p^d+1}}= (-1)^{\frac{m}{d}}$ is equivalent to
 $\displaystyle \alpha^{\frac{q-1}{p^d+1}} =(-1)^{\frac{m}{d}} (1-c)^{-\frac{q-1}{p^d+1}}$.
 With 
 \begin{align*}
 W_a&=\left\{\alpha\in X_a\,:\, A_\alpha^{\frac{q-1}{p^d+1}}\neq (-1)^{\frac{m}{d}} \right\}\\
   \Sigma&=\sum_{\alpha\in X_a} \chi_1(\alpha (P'(a)-b)),\\
 \Sigma_1&=\sum_{\alpha\in X_a\setminus W_a} \chi_1(\alpha (P'(a)-b)),
 \end{align*}
  the sum~\eqref{eq:Tb_even} becomes  {(for even $\frac{n}{d}$)
  \begin{align*}
 T_1&=  (-1)^{\frac{m}{d}} p^m \left(\Sigma-\Sigma_1 \right)
  + (-1)^{\frac{m}{d}+1}\, p^{m+d}  \Sigma_1\\
 &=  (-1)^{\frac{m}{d}} p^m\, \Sigma +  (-1)^{\frac{m}{d}+1} p^m\, \Sigma_1 \left(p^d+1 \right).
 \end{align*}
 }

 We now consider the case of odd $\frac{n}{d}$.   Recall the definition of the Gauss sum
 \[
 G(\psi,\chi)=\sum_{\alpha\in\F_q^*} \psi(\alpha)\chi(\alpha),
 \]
 where $\psi,\chi$ are some multiplicative, respectively, additive characters of $\F_q$.
  We also define an
   incomplete Gauss sum on a set $U\subseteq \F_q^*$ to be $G_U(\psi,\chi)=\sum_{\alpha\in U} \psi(\alpha)\chi (\alpha)$.
   
 Next, when $\frac{n}{d}$ is odd,  $c$ fixed, and  $A=A_\alpha,B=B_\alpha=0,\epsilon_p=  1$,  if $p\equiv 1\pmod 4$, respectively, $ \epsilon_p=\imath^n$, if $ p\equiv 3\pmod 4$, then
  {
 \begin{align*}
 T_1&=  (-1)^{n-1} \epsilon_p \sqrt{q}\, \eta(1-c) \sum_{\alpha\in X_a}\chi_1(\alpha(P'(a)-b))   \eta(\alpha)\\
 &=  (-1)^{n-1} \epsilon_p \sqrt{q}\, \eta(1-c) G_{ X_a} (\eta, \chi_{P'(a)-b}).
 \end{align*}
}

 
 If $\alpha\in\bar X_a$ (so, $B_\alpha\neq 0$), we shall make use of the following result from~\cite{Co98} (we make slight changes in notations and combine various results). 
 \begin{thm}[\textup{\cite{Co98}}]
\label{thm:Co98}
Let $q=p^n$, $n\geq 2$, $p$ an  odd prime, $1\leq k<n$, $d=\gcd(n,k)$. Let $f(x)=A^{p^k} x^{p^{2k}}+Ax$, for some nonzero $A$. The following statements hold\textup{:}
\begin{enumerate}
\item[$(1)$]
If $f$ is a permutation polynomial over $\F_q$, and  $x_0$ is the unique element such that $f(x_0)=-B^{p^k},B\neq 0$, then\textup{:}
\begin{itemize}
\item[$(i)$] If $\frac{n}{d}$ is odd, then 
\[
\mathscr{S}_k(A,B)=
\begin{cases}
(-1)^{n-1} \sqrt{q}\,\eta(-A)\,\overline{\chi_1(Ax_0^{p^k+1})} &\text{ if } p\equiv 1\pmod 4\\
(-1)^{n-1} \imath^{3n} \sqrt{q}\,\eta(-A)\,\overline{\chi_1(Ax_0^{p^k+1})} &\text{ if } p\equiv 3\pmod 4.
\end{cases}
\]
(the solution is $x_0=-\frac{1}{2}\sum_{j=0}^{\frac{n}{d}-1} (-1)^j A^{-\frac{p^{(2j+1)k}+1}{p^k+1}} B^{p^{(2j+1)k}}$).
\item[$(ii)$] If $\frac{n}{d}$ is even, then $n=2m$, $A^{\frac{q-1}{p^d+1}}\neq (-1)^{\frac{m}{d}}$ and
\[
\mathscr{S}_k(A,B)=(-1)^{\frac{m}{d}} p^m\, \overline{\chi_1(Ax_0^{p^k+1})}.
\]
\end{itemize}
\item[$(2)$] If $f$ is not a permutation polynomial, then, for $B\neq 0$, $\mathscr{S}_k(A,B)=0$, unless, $f(x)=-B^{p^k}$ has a solution $x_0$ \textup{(}this can only happen if $\frac{n}{d}$ is even with $n = 2m$, and $A^{\frac{q-1}{p^d+1}}= (-1)^{\frac{m}{d}}$\textup{)}, in which case
\[
\mathscr{S}_k(A,B)=(-1)^{\frac{m}{d}+1}p^{m+d} \overline{\chi_1(Ax_0^{p^k+1})}.
\]
\end{enumerate}
\end{thm}

Let $A=A_\alpha=\alpha(1-c)$, 
  $B=B_\alpha=\sum_{i=0}^{n-1} (a_i')^{p^{n-i}}\neq 0$ (so, $\alpha\in\bar X_a$), 
 where $a_i'= A_\alpha  a_i$, for all $0\neq i\neq k$, $a_0'=\alpha a^{p^k}+A_\alpha a_0$, $a_k'=\alpha  a+ A_\alpha a_k$, and $L_\alpha(x)=A_\alpha^{p^k}x^{p^{2k}}+A_\alpha x$. 
 It is known~\cite{ZWW20}   that a linearized polynomial of the form $L(x)=x^{p^r}+\gamma x\in\F_{p^n}$ is a permutation polynomial (PP) if and only if  $(-1)^{n/e} \gamma^{(p^n-1)/(p^e-1)}\neq 1$, where $e=\gcd(n,r)$. 
 It follows that in our case, with $e=\gcd(n,2k)$, $L_\alpha$ is a PP if and only if 
\[
1\neq (-1)^{\frac{n}{e}} A_\alpha ^{(p^k-1)\frac{p^n-1}{p^e-1}}. 
\] 
From Theorem~\ref{thm:Co98}, if $\frac{n}{d}$ is odd, $\alpha\in\bar X_a$ and $L_\alpha$ is a PP (that is, the above displayed condition holds), then  ($x_\alpha$ is the solution to $L_\alpha(x)=B_\alpha^{p^k}$)
\allowdisplaybreaks
 \begin{align*}
 S_\alpha &= \sum_{x\in \F_q} \chi_1\left(A_\alpha x^{p^k+1}+ B_\alpha   \right)\\
 &=(-1)^{n-1}\mu_p \sqrt{q}\, \eta(-A_\alpha )\,\overline{\chi_1(A_\alpha x_\alpha^{p^k+1})},
\end{align*}
 where $\mu_p=1$, if $p\equiv 1\pmod 4$, and $\mu_p=\imath^{3n}$, if $p\equiv 3\pmod 4$.  
 Recall the
   incomplete Gauss sum on a set $U\subseteq \F_q^*$, namely $G_U(\psi,\chi)=\sum_{\alpha\in U} \psi(\alpha)\chi (\alpha)$.
    Thus, when $\frac{n}{d}$ is odd and $V_a=\{\alpha\in\bar X_a\,:\,A_\alpha^{(p^k-1)\frac{p^n-1}{p^e-1}}\neq (-1)^{\frac{n}{e}} \}$,  then, denoting by $T_2\big|_{V_a}$ the restriction function $T_2$  computed only on $V_a$, and using~\cite[Theorem 5.12 (i)]{LN97} (with $\delta_\alpha =P'(a)-b-(1-c) x_\alpha^{p^k+1}$)
     {
 \allowdisplaybreaks
 \begin{align*}
 T_2\big|_{V_a}&=(-1)^{n-1}\mu_p \sqrt{q}\,\eta(c-1)\\
 &\qquad\times \sum_{x\in V_a}   \eta(\alpha)\,  \chi_1(\alpha(P'(a)-b-(1-c) x_\alpha^{p^k+1}))\\
 &=(-1)^{n-1}\mu_p \sqrt{q}\,\eta(c-1) G_{V_a} (\eta,\chi_{\delta_\alpha }).
 \end{align*}
 }
 Observe that when $\frac{n}{d}$ is odd, then $e=\gcd(n,2k)=\gcd(n,k)=d$ (this is equivalent to  $2^\ell\| n$ and $2^\ell\| k$ for some integer $\ell$, where $2^\ell\|n$ means that $\ell$ is the $2$-valuation of $n$, that is the exact power of $2$ dividing $n$),
  then,
 \[
 A_\alpha^{(p^k-1)\frac{p^n-1}{p^e-1}}=\left(A_\alpha^{\frac{p^k-1}{p^e-1}} \right)^{p^n-1}=1,
 \]
 so, $L_\alpha$ is always a PP (observe that $\frac{n}{d}$ is odd). If that is the case ( $2^\ell\| n$ and $2^\ell\| k$ for some integer $\ell$), then,  
   {
 \begin{align*}
 T_2 &= (-1)^{n-1}\mu_p \sqrt{q}\,\eta(c-1) G_{V_a} (\eta,\chi_{\delta_\alpha }).
 \end{align*}
 }
 
We now consider the case of $\frac{n}{d}$ being even, so $e=2d$.
 If $L_\alpha$ is a PP, and thus, $A_\alpha^{\frac{p^n-1}{p^d+1}}\neq (-1)^{\frac{n}{2d}}$, then
 ($x_\alpha$  is the solution to $L_\alpha(x)=B_\alpha^{p^k}$)
 \begin{align*}
S_\alpha&=  (-1)^{\frac{m}{d}} p^m \chi_1\left(-A_\alpha x_\alpha^{p^k+1}\right),
 \end{align*}
 and when $L_\alpha$ is not a PP (thus, $A_\alpha^{\frac{p^n-1}{p^d+1}}= (-1)^{\frac{n}{2d}}$), but a solution $x_\alpha$ exists to $L_\alpha(x)=B_\alpha^{p^k}$ -- we will call this, condition $(\cP)$, then
 \begin{align*}
S_\alpha&=  (-1)^{\frac{m}{d}+1} p^{m+d} \chi_1\left(-A_\alpha x_\alpha^{p^k+1}\right),
 \end{align*}
 
 As before, we let $V_a=\left\{\alpha\in\bar X_a\,:\, A_\alpha^{\frac{p^n-1}{p^d+1}}\neq (-1)^{\frac{n}{2d}}  \right\}$.
 Putting the previous results together, for   { even $\frac{n}{d}$, we get
 \begin{align*}
 T_2&= (-1)^{\frac{m}{d}} p^{m}\sum_{\alpha\in V_a}  \chi_1\left(\alpha\left(P'(a)-b+(c-1) x_\alpha^{p^k+1}\right)\right)\\
& +(-1)^{\frac{m}{d}+1} p^{m+d}\sum_{\substack{\alpha\in\bar X_a\setminus V_a \\ \alpha \text{ satisfies } (\cP)} } \chi_1\left(\alpha\left(P'(a)-b+(c-1) x_\alpha^{p^k+1}\right)\right).
\end{align*}
}

We thus have shown the following theorem. Let $P(x)=\sum_{i=0}^{n-1} a_i x^{p^i}$ be a linearized polynomial. Recall that for $1\leq k<n$ and $c\in\F_q$, $\alpha\in\F_q^*$, $d=\gcd(n,k), e=\gcd(n,2k)$, we let $P'(a)=P(a)+a^{p^k+1}$, $A_\alpha=\alpha(1-c)$, $B_\alpha=\sum_{i=0}^{n-1} (a_i')^{p^{n-i}}$, 
 where $a_i'= A_\alpha  a_i$, for all $0\neq i\neq k$, $a_0'=\alpha  a^{p^k}+A_\alpha  a_0$ and $a_k'=\alpha (a+(1-c)a_k)$. Further, for fixed $a\in \F_q$, 
 \allowdisplaybreaks
 \begin{align*}
 X_a&=\left\{\alpha\in\F_q^*\,:\,  \alpha a^{p^k}+\alpha^{p^{-k}} a +P^*(\alpha)=0 \right\},\\
 V_a&=\left\{\alpha\in\bar X_a\,:\, A_\alpha^{\frac{p^n-1}{p^d+1}}\neq (-1)^{\frac{n}{2d}}  \right\},\\
 W_a&=\left\{\alpha\in X_a\,:\, A_\alpha^{\frac{q-1}{p^d+1}}\neq (-1)^{\frac{n}{2d}} \right\},\\
   \Sigma&=\sum_{\alpha\in X_a} \chi_1(\alpha (P'(a)-b)),\\
 \Sigma_1&=\sum_{\alpha\in X_a\setminus W_a} \chi_1(\alpha (P'(a)-b)).
 \end{align*}
  We also define an
   incomplete Gauss sum on a set $U\subseteq \F_q^*$, namely, $G_U(\psi,\chi)=\sum_{\alpha\in U} \psi(\alpha)\chi (\alpha)$, and $\delta_\alpha=P'(a)-b-(1-c) x_\alpha^{p^k+1}$, where $x_\alpha$ is the solution of the equation $L_\alpha(x)=B_\alpha^{p^k}$ (we called this, condition $(\cP)$). Let also, $\mu_p=1$, if $p\equiv 1\pmod 4$, and $\mu_p=\imath^{3n}$, if $p\equiv 3\pmod 4$.
\begin{thm}
\label{thm:main2}
Let $F(x)=x^{p^k+1}$ $(p$ is an odd prime, $n\geq 2$, and $k<n)$ be the Gold function, $P(x)=\sum_{i=0}^{n-1} a_i x^{p^i}$ be a linearized polynomial and $c\in\F_{p^n}$.  Then, the $c$-Differential Distribution Table entries  of $G(x)=F(x)+P(x)$ at  $a,b\in\F_{p^n}$ are given by ${_c}\Delta_G(a,b)=1+p^{-n}  T_{a,b}$, where\textup{:}
\begin{enumerate}
\item[$(i)$] Let   $\frac{n}{d}$ be even, $n=2m$. Then
\begin{align*}
T_{a,b}&= (-1)^{\frac{m}{d}} p^m\, \Sigma +  (-1)^{\frac{m}{d}+1} p^m\, \Sigma_1 \left(p^d-1 \right)\\
&+(-1)^{\frac{m}{d}} p^{m}\sum_{\alpha\in V_a}  \chi_1\left(\alpha\delta_\alpha\right)
 +(-1)^{\frac{m}{d}+1} p^{m+d}\sum_{\substack{\alpha\in\bar X_a\setminus V_a \\ \alpha \text{ satisfies } (\cP)} } \chi_1\left(\alpha\delta_\alpha\right).
\end{align*}
 \item[$(ii)$] Let  $\frac{n}{d}$ be odd. Then
\begin{align*}
 T_{a,b}&= (-1)^{n-1} \mu_p p^{\frac{n}{2}}\, \eta(1-c) G_{ X_a} (\eta, \chi_{P'(a)-b})\\
 &\qquad +(-1)^{n-1}\mu_p  p^{\frac{n}{2}}\,\eta(c-1) G_{V_a} (\eta,\chi_{\delta_\alpha }).
\end{align*}
 \end{enumerate}
\end{thm}
The following corollary is immediate.
\begin{cor}
With the notations of Theorem~\textup{\ref{thm:main2}}, we have\textup{:}
\begin{enumerate}
\item[$(i)$] If  $\frac{n}{d}$ is even, then
\[
{_c}\Delta_G(a,b)\leq 1+p^{-\frac{n}{2}} |X_a|+p^{-\frac{n}{2}} (p^d-1)|X_a\setminus W_a|+p^{-\frac{n}{2}} |V_a|+p^{-\frac{n}{2}+d}|\bar X_a\setminus V_a|.
\]
\item[$(ii)$]If  $\frac{n}{d}$ is even, then
\[
{_c}\Delta_G(a,b)\leq 1+p^{-\frac{n}{2}} \left(|X_a|+|V_a|\right).
\]
\end{enumerate}
\end{cor}

\begin{rem}
We can bring more light into  Equation~\eqref{eq:X_asol}, if we were to consider it as an equation in $a$, not $\alpha$.
We know~\textup{\cite{ZWW20}}   that a linearized polynomial of the form $L(x)=x^{p^r}+\gamma x\in\F_{p^n}$ is a permutation polynomial if and only if the relative norm $N_{\F_{p^n}/\F_{p^d}}(\gamma)\neq 1$, that is, $(-1)^{n/d} \gamma^{(p^n-1)/(p^d-1)}\neq 1$, where $d=\gcd(n,r)$. In our case, $r=2k$ and  $\gamma=\alpha^{1-p^k}$ \textup{(}the condition can be written in terms of $\gamma$, or $\gamma^{-1}$\textup{)}, and so, for fixed $\alpha\neq 0$ if $\alpha^{p^k} x^{p^{2k}} +\alpha x$ is a PP (that is, $1\neq (-1)^{\frac{n}{e}} \alpha ^{\frac{p^n-1}{p^k+1}}$), there is a unique root $a$ of the above equation.  

We can also do the general case, when perhaps the previous linearized polynomial is not a PP, by using~\textup{\cite{CH04}}. With $t=\frac{n}{\gcd(2k,n)}$, and the notations of~\textup{\cite{CH04}}, we let
\begin{align*}
\alpha_{t-1}:&=(-\alpha^{1-p^k})^{1+p^{2k}+\cdots+p^{2k(t-1)}}=(-1)^t \alpha^{\frac{1-p^{2kt}}{p^k+1}},\\
\beta_{t-1}:&=\sum_{i=0}^{t-2}(-\alpha^{1-p^k})^{\sum_{j=i}^{t-2}p^{2k(j+1)}}\left(-\frac{P^*(\alpha)}{\alpha}\right)^{p^{2ki}}+\left(-\frac{P^*(\alpha)}{\alpha}\right)^{p^{2k(t-1)}}\\
&=\sum_{i=0}^{t-2} (-1)^{t-i} \alpha^{  \frac{p^{2k(i+1)}-p^{2kt}}{p^k+1}}\left(\frac{P^*(\alpha)}{\alpha}\right)^{p^{2ki}}-\left(\frac{P^*(\alpha)}{\alpha}\right)^{p^{2k(t-1)}}\\
&=\sum_{i=0}^{t-2} (-1)^{t-i} \alpha^{  \frac{p^{2k(i+1)}-p^{2kt}}{p^k+1}-p^{2ki}}\left(P^*(\alpha)\right)^{p^{2ki}} -\left(\frac{P^*(\alpha)}{\alpha}\right)^{p^{2k(t-1)}}\\
&= \alpha^{  \frac{-p^{2kt}}{p^k+1}}\sum_{i=0}^{t-2} (-1)^{t-i}  \left(\alpha^{\frac{p^{2k}}{p^k+1}-1}P^*(\alpha)\right)^{p^{2ki}} -\left(\frac{P^*(\alpha)}{\alpha}\right)^{p^{2k(t-1)}} \\
&=(-1)^t  \alpha^{  \frac{-p^{2kt}}{p^k+1}}\sum_{i=0}^{t-1} (-1)^{i}  \left(\alpha^{\frac{p^{2k}}{p^k+1}-1}P^*(\alpha)\right)^{p^{2ki}}.
\end{align*}
\textup{(}Though the final expressions are in $\F_q$, we regard the various terms in this last identity to belong in an extension of $\F_q$, otherwise a factor like $\alpha^{\frac{p^{2k}}{p^k+1}-1}$ makes little sense, for some $\alpha$'s.\textup{)}

 By~\textup{\cite{CH04}}, if $\alpha_{t-1}=1$ and $\beta_{t-1}\neq0$, there are no solutions $a$ for Equation~\eqref{eq:X_asol}\textup{;} if $\alpha_{t-1}=1$ and $\beta_{t-1}=0$, there are $p^d$ solutions\textup{;} and, if $\alpha_{t-1}\neq1$, there is one solution.
\end{rem}
 
 \section{Entries of the $c$-DDT for the perturbed Gold function via a linearized polynomial, $p$ even}
 \label{sec4}
 
 Here $q=2^n$. We let as before $\mathscr{S}_k(A,B)=\sum_{x\in \F_q} \chi_1\left(A x^{p^k+1}+ B x\right)$, and $A_\alpha,B_\alpha$ defined as in the previous section. In this case, using~\cite{Co99}, we have that, if $\frac{n}{d}$ is odd, where $d=\gcd(n,k)$,
 \[ 
 \mathscr{S}_k(A_\alpha,B_\alpha)=
 \begin{cases}
 0& \mbox{ if } \Trn(B_\alpha C_\alpha^{-1})\neq1\\
 \pm2^{\frac{n+d}{2}}& \mbox{ if } \Trn(B_\alpha C_\alpha^{-1})=1.
 \end{cases}
 \]
 where $C_\alpha$ is the only element such that $C_\alpha^{2^k+1}=A_\alpha$ (by Lemma~\ref{lem:gcd}, $\gcd(2^k+1,2^n-1)=1$, when $\frac{n}{d}$ is odd).
 
 In \cite{Co99}, combining Lemma 4.3 and Theorem 4.6, we see further that:
 $\mathscr{S}_k(1,1)=\left(\frac{2}{n/d}\right)^d2^{\frac{n+d}{2}}$, where $\left(\frac{2}{s}\right)$ is the Jacobi symbol, and 
 $\mathscr{S}_k(A,B)=\chi_1(\gamma^{2^k+1}+\gamma)\mathscr{S}_k(1,1)$, with $B_\alpha C_\alpha^{-1}=\gamma^{2^{2k}}+\gamma+1$, for some $\gamma\in\F_q$.
  {In conclusion, for $\frac{n}{d}$ odd, with $C_\alpha$ and $\gamma$ as before, and denoting by $W=\left\{\alpha : \Trn(B_\alpha C_\alpha^{-1})=1\right\}$ and $  \Sigma_2=\sum_{\alpha\in W} \chi_1(\alpha (P'(a)-b)+\gamma^{2^k+1}+\gamma)$, then
 \[
 {_c}\Delta_G(a,b)= 1+\left(\frac{2}{n/d}\right)^d2^{\frac{d-n}{2}} \Sigma_2.
 \]
 
 } 
 
 For $\frac{n}{d}$ even, we use Theorem 5.3 of \cite{Co99}, which we cite here for the convenience of the reader.
 \begin{thm}[\textup{\cite{Co99}}]
 Let $B\in\F_q^*$, $q=2^n$, $g$ be a primitive element of $\F_q$, and suppose that $\frac{n}{d}$ is even so that $n=2m$ for some integer $m$. Let $f(x)=A^{2^k}x^{2^{2k}}+Ax$. The following statements hold\textup{:} 
 \begin{itemize}
\item[$(i)$] If $A\neq g^{t(2^d+1)}$ for some integer $t$ then $f$ is a PP. Let $x_0\in\F_q$ be the unique element satisfying $f(x_0)=B^{2^k}$. Then,
\[
\mathscr{S}_k(A,B)=(-1)^{\frac{m}{d}}2^m\chi_1\left(Ax_0^{2^k+1}\right).
\]
\item[$(ii)$] If $A= g^{t(2^d+1)}$ for some integer $t$ then $\mathscr{S}_k(A,B)=0$ unless the equation $f(x)=B^{2^k}$ is solvable. If the equation is solvable, with solution $x_0$, say, then
\[
\mathscr{S}_k(A,B)=
\begin{cases}
(-1)^{\frac{m}{d}+1}2^{m+d}\chi_1\left(Ax_0^{2^k+1}\right)&\mbox{ if } \Trd(A)=0\\
(-1)^{\frac{m}{d}}2^{m}\chi_1\left(Ax_0^{2^k+1}\right)&\mbox{ if } \Trd(A)=0.
\end{cases}
\]
\end{itemize}
 \end{thm}
  {Then, if $\frac{n}{d}$ is even, so that $n=2m$, and denoting by 
 \begin{align*}
 Y&=\{\alpha: A_\alpha\neq g^{t(2^d+1)}\}, \\
 Z_1&=\{\alpha: A_\alpha= g^{t(2^d+1)}, f(x)=B_\alpha^{2^k}\mbox{ is solvable, and }\Tr_d(A_\alpha)\neq0\},\\ 
 Z_2&=\{\alpha: A_\alpha= g^{t(2^d+1)}, f(x)=B_\alpha^{2^k}\mbox{ is solvable, and }\Tr_d(A_\alpha)=0\},
 \end{align*} 
 then
 \allowdisplaybreaks
 \begin{align*}
 {_c}\Delta_G(a,b)&= 1+(-1)^{\frac{m}{d}}2^{-m}\left(\sum_{Y\cup Z_1}\chi_1\left(A_\alpha x_\alpha^{2^k+1}+\alpha(P'(a)-b)\right)\right.\\
 &\left.-2^d\sum_{Z_2}\chi_1\left(A_\alpha x_\alpha^{2^k+1}+\alpha(P'(a)-b)\right)\right).
 \end{align*}
 Denoting further, for any set $U$, by $\Sigma_U=\sum_{U}\chi_1\left(A_\alpha x_\alpha^{2^k+1}+\alpha(P'(a)-b)\right)$, then
  \[
 {_c}\Delta_G(a,b)= 1+(-1)^{\frac{m}{d}}2^{-m}\left(\Sigma_{Y\cup Z_1}-2^d\Sigma_{Z_2}\right).
 \]
 } 
 We have then proven the following theorem (with the above notations).
 \begin{thm}
 \label{thm:main3}
 Let $G(x)=x^{2^k+1}+P(x)$ be a perturbation of the Gold function on $\F_{2^n}$ (of primitive element $g$), where $P$ is a linearized polynomial, and $d=\gcd(n,k)$. For each $\alpha\in\F_{2^n}^*$, we let $f_\alpha(x)=A_\alpha^{2^k}x^{2^{2k}}+A_\alpha x$, where $A_\alpha,B_\alpha$ are defined in Section~\textup{\ref{sec3}}. Then, the $c$-Differential Distribution Table entries  of $G(x)$ at  $a,b\in\F_{2^n}$ are given by\textup{:}
  \begin{itemize}
 \item[$(i)$] If $\frac{n}{d}$ is odd,  $C_\alpha= A_\alpha ^{\frac{1}{2^k+1}}$, $B_\alpha C_\alpha^{-1}=\gamma^{2^{2k}}+\gamma+1$, for some $\gamma\in\F_{2^n}$, $W=\left\{\alpha : \Trn(B_\alpha C_\alpha^{-1})=1\right\}$ and $  \Sigma_2=\sum_{\alpha\in W} \chi_1\left(\alpha \left(P'(a)-b\right)+\gamma^{2^k+1}+\gamma\right)$,
 \[
 {_c}\Delta_G(a,b)= 1+\left(\frac{2}{n/d}\right)^d2^{\frac{d-n}{2}} \Sigma_2,
 \]
 where $\left(\frac{2}{s}\right)$ is the Jacobi symbol.
 \item[$(ii)$] If $\frac{n}{d}$ is even, so that $n=2m$, and denoting by $Y=\left\{\alpha: A_\alpha\neq g^{t(2^d+1)}\right\}$,  $Z_1=\left\{\alpha: A_\alpha= g^{t(2^d+1)}, \Tr_d(A_\alpha)\neq0, f_\alpha(x)=B_\alpha^{2^k}\mbox{ is solvable}\right\}$, $Z_2=\left\{\alpha: A_\alpha= g^{t(2^d+1)}, \Tr_d(A_\alpha)=0, f_\alpha(x)=B_\alpha^{2^k}\mbox{ is solvable}\right\}$, and, for any set $U$, letting $\Sigma_U=\sum_{U}\chi_1\left(A_\alpha x_\alpha^{2^k+1}+\alpha(P'(a)-b)\right)$,
 \[
 {_c}\Delta_G(a,b)= 1+(-1)^{\frac{m}{d}}2^{-m}\left(\Sigma_{Y\cup Z_1}-2^d\Sigma_{Z_2}\right).
 \]
 \end{itemize}
 \end{thm} 
 
 The following corollary is immediate.
\begin{cor}
With the notations of Theorem~\textup{\ref{thm:main3}}, we have\textup{:}
\begin{enumerate}
\item[$(i)$] If  $\frac{n}{d}$ is even, then
\[
{_c}\Delta_G(a,b)\leq 1+ 2^{\frac{d-n}{2}} \left|\left\{\alpha : \Trn(B_\alpha C_\alpha^{-1})=1\right\} \right|.
\]
\item[$(ii)$]If  $\frac{n}{d}$ is even, then
\[
{_c}\Delta_G(a,b)\leq 1+ 2^{-\frac{n}{2}}\left(|Y\cup Z_1|+2^d|Z_2|\right).
\]
\end{enumerate}
\end{cor}

\section{Computational results}

In this section, we give the maximal $c$-differential uniformity over $\F_{2^n}$ for the concrete Gold perturbation $G(x)=x^{2^k+1}+x^{2^i}+x^{2^j}$,  for $2\leq n\leq 6$, and all $0\leq i<j<n$, $1\leq k<n$. We will also include (for comparison purposes) the $c$-differential uniformity ($c$DU) for the Gold function under the row~$00$ (in~\cite{RS20}, the $c$-differential uniformity of the Gold function is completely described when $\gcd(n,k)=1$ and also when $\gcd(n,k)>1$, under some technical conditions). We shall denote by $\beta_G=\max_{c\neq 1} {_c}\beta_G$. For $n=2$, $\beta_G=3$, which the same as for the Gold function. From the tables below we see that the  $c$-differential uniformity  of the perturbation fluctuates, in some instances being three times as much, e.g, $n=6,k=1, (i,j)=(3,4)$. Furthermore, there are cases when it does decrease,  e.g, $n=6,k=3, (i,j)=(2,3)$.

\vskip-.25cm 
\captionsetup{width=4.0cm,position=top}

\begin{table}[h]
\footnotesize
\floatbox[\capbeside]{table}
{\caption{\small Maximal $c$-differential uniformity $\beta_G$, for $n=3$}}
{\flushleft
\begin{tabular}{|c|c|c|}
\hline
$(i,j)$ & $k=1$ & $k=2$  \\ \hline 
(0,0) & 3 &3\\ \hline
(0,1) & 3 & 4 \\ \hline
(0,2) & 4 & 3\\ \hline
(1,2) & 4 & 4\\ \hline
\end{tabular}}
\end{table}

 
\begin{table}[!htbp]
\footnotesize
\floatbox[\capbeside]{table}
{\caption{\small Maximal $c$-differential uniformity $\beta_G$, for $n=4$}}
{\flushleft
\begin{tabular}{|c|c|c|c|}
\hline
$(i,j)$ & $k=1$ & $k=2$ & $k=3$ \\ \hline 
(0,0) &  3 & 5 &3 \\ \hline
(0,1) & 3 &  5&  4 \\ \hline
(0,2) &  4 & 5 & 6 \\ \hline
(0,3) &  6 & 5  & 3 \\ \hline
(1,2) & 4  & 5  & 5 \\ \hline
(1,3) &  6 & 5  & 4 \\ \hline
(2,3) & 5 &  5 & 6 \\ \hline
\end{tabular}}
\end{table}


\begin{table}[!htbp]
\footnotesize
\floatbox[\capbeside]{table}
{\caption{\small Maximal $c$-differential uniformity $\beta_G$, for $n=5$}}
{\flushleft
\begin{tabular}{|c|c|c|c|c|}
\hline
$(i,j)$ & $k=1$ & $k=2$ & $k=3$ & $k=4$ \\ \hline 
(0,0) & 3 & 3 & 3 &3 \\ \hline
(0,1) & 3 &5 & 5 & 4 \\ \hline
(0,2) & 4 & 3& 6 & 6 \\ \hline
(0,3) & 6 & 5 & 3& 6\\ \hline
(0,4) & 6  &  6 & 5& 3 \\ \hline
(1,2) & 4 & 5  & 6&7\\ \hline
(1,3) & 6 & 7  &5 &6 \\ \hline
(1,4)& 6 &6 &7 &4  \\ \hline
(2,3)& 7 & 5&6 &5 \\ \hline
(2,4)& 6 &6 & 6& 6\\ \hline
(3,4)& 5 & 6& 5& 6\\ \hline
\end{tabular}}
\end{table}

  \vskip-.3cm
  
\begin{table}[!htbp]
\footnotesize
\floatbox[\capbeside]{table}
{\caption{\small  Maximal $c$-differential uniformity $\beta_G$, for $n=6$}}
{
\begin{tabular}{|c|c|c|c|c|c|}
\hline
$(i,j)$ & $k=1$ & $k=2$ & $k=3$ & $k=4$ & $k=5$ \\ \hline 
(0,0) & 3 & 5& 9&5  & 3\\ \hline
(0,1) & 3 & 5&9  & 5& 4 \\ \hline
(0,2) & 4 & 5&6  &5 & 7 \\ \hline
(0,3) &7  & 5 & 9 &10 & 7\\ \hline
(0,4) & 7 &  5 & 9& 5& 6\\ \hline
(0,5) & 6 & 10  &6 & 5& 3 \\ \hline
(1,2) &4  &  5 &8 &7 & 6\\ \hline
(1,3) &7  & 8  & 9& 5& 6\\ \hline
(1,4)& 7 & 6 &15 &5 & 8\\ \hline
(1,5) & 6  &  10 & 6 & 8& 4 \\ \hline
(2,3)& 6 & 5& 6 & 10& 9 \\ \hline
(2,4)& 6 &5 &6 & 5&7 \\ \hline
(2,5) & 8 &10   & 13 & 6&7 \\ \hline
(3,4)& 9 & 7&9 &10 &7 \\ \hline
(3,5)& 7 &  5 & 6& 10& 7 \\ \hline
(4,5)&  7&  10 &8 & 5& 6 \\ \hline
\end{tabular}}
\end{table}

  \section{Concluding remarks}
\label{sec5}

In this paper we first show that a perturbation (it is known~\cite{SPRS20} that the $c$-differential uniformity is not invariant under EA or CCZ equivalence) of the Gold function via a linearized monomial has the property that its $c$-differential uniformity spectrum tends to increase significantly for some $c$. We further  
 propose a new approach for the computation of  the $c$-DDT entries and the $c$-differential  uniformity via characters in the finite field.  We then apply our method for the Gold function perturbed by any linearized polynomial. It is the first such investigation providing exact expressions for the full $c$-DDT table (albeit, in terms of characters on the finite field). 
 We provide detailed computations for the $c$-differential uniformity of a perturbation of the Gold function via linearized binomials, for small dimensions.
 We further propose here that one could look at perturbations of  other PN/APN functions under EA-transformations and investigate their $c$-differential uniformity.

 \end{document}